\documentclass[reqno,draft]{amsart}
\usepackage{amssymb, amscd, amsfonts}
\usepackage[mathscr]{eucal}
\usepackage{graphicx}
\usepackage{amscd}

\newtheorem{theorem}{Theorem}[section]
\newtheorem{lemma}[theorem]{Lemma}
\newtheorem{corollary}[theorem]{Corollary}
\newtheorem{proposition}[theorem]{Proposition}

\theoremstyle{definition}
\newtheorem{definition}[theorem]{Definition}

\newcommand{\restrict}{\,{\mathbin{\vert\mkern-0.3mu\grave{}}}\,}

\newcommand{\remove}[1]{}

\DeclareMathOperator{\McNn}{\mathscr M_{\it n}}
\DeclareMathOperator{\conv}{\rm conv}
\DeclareMathOperator{\den}{\rm den}

\DeclareMathOperator{\Zed}{\mathbb{Z}}

\DeclareMathOperator{\cube}{[0,1]^{\it n}}

\DeclareMathOperator{\maxspec}{\rm MaxSpec}
\DeclareMathOperator{\apo}{\rm apo}
\DeclareMathOperator{\apor}{\rm apo_{\mathbb R}}

\DeclareMathOperator{\sgr}{\rm sgr}
\DeclareMathOperator{\grp}{\rm grp}

\title[Finitely presented unital $\ell$-groups]{\bf Finitely
presented  lattice-ordered abelian groups with order-unit}

\author[L.Cabrer]{Leonardo Cabrer $^\ddag$}
\address[L.Cabrer]{Research Center for Integrated Sciences \\
Japan Advanced Institute of Sciences an Technology \\
1-1 Asahidai -- Nomi -- Ishikawa \\
Japan }
\email{lmcabrer@yahoo.com.ar }

\author[D.Mundici]{Daniele Mundici$^\dag$}
\address[D.Mundici]{Dipartimento di
Matematica \, ``Ulisse Dini'' \\
Universit\`{a} degli Studi di Firenze \\
viale Morgagni 67/A \\
I-50134 Firenze \\
Italy}
\email{mundici@math.unifi.it }

\keywords{Lattice-ordered abelian group, order-unit,
spectral space, basis,   Schauder basis,
dimension group,
Elliott classification, simplicial group,
finite presentation, projective,
rational polyhedron, simplicial complex, unimodular
triangulation, piecewise-linear function,
regular fan, nonsingular fan,
retraction.}

\subjclass[2000]{Primary: 06F20.
Secondary:    08B30,   14M25,
 20F60,     52B20.}

\date{\today}
\begin{document}

\begin{abstract}
Let $G$ be an $\ell$-group (which is short for
``lattice-ordered abelian group'').  Baker and Beynon proved that $G$
is finitely presented iff it is finitely generated and projective.  In
the category $\mathcal U$ of {\it unital} $\ell$-groups---those
$\ell$-groups having a distinguished order-unit $u$---only the
$(\Leftarrow)$-direction holds in general.  Morphisms in $\mathcal U$
are {\it unital $\ell$-homomorphisms,} i.e., hom\-o\-mor\-phisms that
preserve the order-unit and the lattice structure.  We show that a
unital $\ell$-group $(G,u)$ is finitely presented iff it has a basis,
i.e., $G$ is generated by an abstract Schauder basis over its maximal
spectral space.  Thus every finitely generated projective unital
$\ell$-group has a basis $\mathcal B$.  As a partial converse, a large
class of projectives is constructed from bases satisfying
$\bigwedge\mathcal B\not=0$.  Without using the Effros-Handelman-Shen
theorem, we finally show that
the bases of any finitely presented unital $\ell$-group
$(G,u)$  provide a direct system of simplicial groups with
1-1 positive unital homomorphisms, whose limit is $(G,u)$.
\end{abstract}

\maketitle

\section{Introduction}
We refer to \cite{bigkeiwol},   \cite{gla}
 and \cite{goo2}
 for background on
$\ell$-groups.
 A unital $\ell$-group $(G,u)$ is an abelian group $G$ equipped with a
translation-invariant lattice-order and a distinguished
{\it order-unit}
$u$, i.e., an element whose positive integer multiples eventually
dominate each element of $G$. Unital $\ell$-groups are a
modern mathematization of the time-honored euclidean
magnitudes with an archimedean
unit (see \cite{mar}).
By \cite[Theorem 3.9]{mun86},
the  category $\mathcal U$ of unital $\ell$-groups
is  equivalent to the equational class of MV-algebras.
Thus, while the archimedean property of order-units is not
definable in first-order logic, $\mathcal U$ is endowed with all
typical properties of equational classes: in particular,
$\mathcal U$ has subalgebras, quotients and products, which in general
are not cartesian products, \cite{dvuhol}.

Finitely presented $\ell$-groups  (with or without unit)
are an active topic of current research, because
they have a basic proteiform reality, as computable
algebraic  structures,  rational polyhedra, fans,
finitely axiomatizable
theories in many-valued logic,
and finitely presented AF C$^{*}$
algebras whose Murray-von Neumann order of projections is a lattice.
See \cite{glapoi, mun88, marmun, manmarmun, mun08}.

Morphisms in the category of $\ell$-groups are
lattice-preserving homomorphisms.
Morphisms in the category of unital $\ell$-groups
are also required to preserve  order-units.
A  unital $\ell$-group
$(G,u)$ is {\it projective} if whenever $\psi\colon (A,a)\to(B,b)$ is
a surjective morphism and $\phi\colon (G,u)\to(B,b)$ is a morphism,
there is a morphism $\theta\colon (G,u)\to(A,a)$ such that $\phi= \psi
\circ \theta$.
For $\ell$-groups, Baker \cite{bak} and Beynon
 \cite{bey75}, \cite[Theorem 3.1]{bey77} (also see \cite[Corollary
 5.2.2]{gla}) gave the following characterization: {\it An $\ell$-group
 $G$ is finitely generated projective iff it is finitely presented.}
 For unital $\ell$-groups the $(\Rightarrow)$-direction holds
 (\cite[Proposition 5]{mun08}).  The converse direction fails in
 general.

Schauder bases provide the main tool to prove that
an   $\ell$-group is finitely generated projective
iff it is presented by a word in the pure language of
lattices, without resorting to the group structure,
\cite{manmarmun}.
This strengthens the  characterization by Baker-Beynon
mentioned above,
where lattice-group words are used,
and paves the way to a full understanding of the sharp
differences between measure theory in unital and in
non-unital $\ell$-groups, \cite{mun08}.

For a geometric investigation of
finitely presented unital $\ell$-groups,
in \cite{marmun}  the notion of
{\emph basis} (see Definition \ref{def:basis})
was introduced as a purely algebraic counterpart of
Schauder bases.
Specifically, in  \cite[Theorem 4.5]{marmun}
it is proved that an {\it archimedean}
unital $\ell$-group $(G,u)$ is
 finitely presented
iff it has a basis.  The archimedean condition means
that $G$ is isomorphic to an $\ell$-group of
real-valued functions defined on some set $X$.
In Theorem \ref{theorem:basis}
we will prove that the archimedean assumption
can be dropped, thus obtaining  a characterization
of finitely presented unital $\ell$-groups that does
not mention free objects and their universal property.
 \smallskip

As a corollary, every  finitely generated
projective unital $\ell$-group  has a basis.
In Section \ref{sec:proj} we will prove
a partial converse,  yielding a method to
construct  large classes of projective unital $\ell$-groups.

\smallskip

With reference to
\cite{fuc-pisa} and \cite{goo1},  the
underlying dimension group of $(G,u)$ will be considered
in the final section.
In  Theorem \ref{theorem:system} it is proved that
 if $(G,u)$ has a basis then
its bases   provide a direct system of
simplicial groups with 1-1 positive unital homomorphisms,
whose limit is $(G,u)$. Thus the  Effros-Handelman-Shen
representation theorem \cite{effhanshe},
Grillet's theorem \cite[2.1]{gri},   and Marra's theorem
\cite{mar} have a very simple  proof for any
such $(G,u)$.

\section{Preliminaries}
\subsection{Unital $\ell$-groups and bases}

 A {\it lattice-ordered abelian group} ({\it $\ell$-group}) is a structure
$(G,+,-,0,\vee,\wedge)$ such that $(G,+,-,0)$ is an abelian group,
$(G,\vee,\wedge)$ is a lattice, and $x+(y\vee z)=(x+y)\vee(x+z)$ for
all $x,y,z\in G$.  An {\it order-unit} in $G$ is an element $u\in G$ with the property that for
every $g\in G$ there is $ n\in\{1,2,3,\ldots\}$ such that $g\leq nu.$ A
{\it unital $\ell$-group} $(G,u)$ is an $\ell$-group $G$ with a
distinguished order-unit $u$.

A map $h\colon (G,u)\rightarrow (G',u')$ is said
to be a {\it unital $\ell$-homomorphism} if it
preserves the lattice as well as the group structure,
and $h(u)=u'$.
By an {\it ideal} $\mathfrak i$ of a unital
$\ell$-group $(G,u)$ we mean the kernel of a
unital $\ell$-homomorphism of $(G,u)$.
 We denote by $\maxspec(G,u)$ the set of maximal
 ideals of $(G,u)$ equipped with the {\it spectral} topology, \cite[\S
10]{bigkeiwol}: a basis of closed sets for
$\maxspec(G)$ is given by sets of the form $ \{\mathfrak
p \in \maxspec(G) \mid a\in \mathfrak p\}, $ where $a$ ranges over all
elements of $G$. Since $G$ has an order-unit,
$\,\,\,\maxspec(G)$ is a nonempty compact Hausdorff space,
\cite[10.2.2]{bigkeiwol}.

\begin{definition}\label{def:basis}
Let $(G,u)$ be a unital $\ell$-group.  A {\em basis} of $(G,u)$ is a
set $\mathcal B = \{b_{1},\ldots,b_{n}\}$ of nonzero elements of the
positive cone $G^{+}=\{g\in G \mid g\geq 0\}$ such
that
\begin{itemize}
    \item[(i)] $\,\mathcal B$ generates $G$;

    \item[(ii)]
    for each $k=1,2,\ldots$ and $k$-element subset $C$ of $\mathcal B$
with  $0\not= \bigwedge\{b\mid b\in C\}$, the set
$\{\mathfrak{m}\in\maxspec(G)\mid   \mathfrak{m}
\supseteq  \mathcal{B}\setminus C \}$
is homeomorphic to a $(k-1)$-simplex;

\item[(iii)] there are
  integers
$1 \leq m_{1},\ldots,m_{n}$ such that
$\,\sum_{i=1}^{n} \,\, m_{i}b_{i} = u$.
  \end{itemize}

\end{definition}
  This is an equivalent simplified reformulation
of \cite[Definition 4.3]{marmun}.
{}From (ii)-(iii) it follows that
the {\em mul\-ti\-pli\-city}
$\,m_{i}$  of each  $b_{i}\in \mathcal B$
is uniquely determined.

 For $n=1,2,\ldots$ we let ${\mathscr M}_{n}$ denote the
unital $\ell$-group of all continuous functions $f\colon [0,1]^{n}\to
\mathbb R$ having the following property: there are (affine) linear
polynomials $p_{1},\ldots,p_{m}$ with integer coefficients, such that
for all $x\in [0,1]^{n}$ there is $i\in \{1,\ldots,m\}$ with
$f(x)=p_{i}(x)$.  ${\mathscr M}_{n}$ is equipped with the pointwise
operations $+,-,\wedge,\vee$ of $\mathbb R$, and with the constant
function 1 as the distinguished order-unit.
The characteristic universal property of
${\mathscr M}_{n}$ is as follows:

\begin{proposition}
\label{proposition:free}
{\rm (\rm \cite[4.16]{mun86})}  ${\mathscr
M}_{n}$ is generated by the coordinate maps $\pi_ {i} \colon
[0,1]^{n}\to \mathbb R$ together with the order-unit $1$.
For every unital $\ell$-group $(G,u)$ and elements
$g_{1},\ldots,g_{n}$ in the unit interval $[0,u]$ of $G$, if
$g_{1},\ldots,g_{n}$ and $u$ generate $G$,
then there is a unique unital
$\ell$-homomorphism $\psi$ of ${\mathscr M}_{n}$ onto $G$ such that
$\psi(\pi_ {i})=g_{i}$ for each $i=1,\ldots,n.$
\end{proposition}

We say that
$(G,u)$ is {\it finitely
presented} if for some $n=1,2,\ldots,$
$(G,u)$ is isomorphic to the quotient of  ${\mathscr
M}_{n}$ by a finitely generated (= singly generated =
principal) ideal.

Given $f\in\McNn$ we deonte ${\mathcal Z}f= a^{-1}(0)$ the
{\it zeroset} of $f$.  More
 generally, for every ideal $\mathfrak j$ of $\McNn$ we will write
 \begin{equation}
     \label{equation:zeroset}
 {\mathcal Z}\mathfrak j = \bigcap \{{\mathcal Z}g\mid g\in \mathfrak j\}.
\end{equation}
 In the particular case when $\mathfrak j$
 is maximal,  ${\mathcal Z}\mathfrak j$ is a singleton
 (because the functions in $\McNn$ separate points,
 \cite[4.17]{mun86}),  and we write
  \begin{equation}
     \label{equation:dotzeroset}
 {\dot {\mathcal Z}}\mathfrak j =
 \mbox{ the only element of }  {\mathcal Z}\mathfrak j.
\end{equation}

For later use we record here a classical result,
 whose proof follows
 from the Hion-H{\"o}lder theorem \cite[p.45-47]{fuc},
 \cite[2.6]{bigkeiwol}:

 \begin{lemma}
     \label{lemma:hoelder}
     For every unital $\ell$-group
 $(G,u)$ and ideal $\mathfrak m\in \maxspec G$ there is exactly one
 pair $(\iota_{\mathfrak m },R_{\mathfrak m })$ where $R_{\mathfrak m
 }$ is a unital $\ell$-subgroup of $(\mathbb R,1)$, and
 $\iota_{\mathfrak m }$ is a unital $\ell$-isomorphism of the quotient
 $(G,u)/\mathfrak m$ onto $R_{\mathfrak m }$.  Upon identifying
 $(G,u)/\mathfrak m$ with $R_{\mathfrak m }$ every element $g/\mathfrak
 m\in (G,u)/\mathfrak m$ becomes a real number, and we can
 unambiguously write $ g/\mathfrak m \in \mathbb R. $
\end{lemma}

\begin{corollary}\label{corollary:HolI}
Let $\mathfrak{i}$ be an ideal of $\McNn$ and $\maxspec_{\supseteq \mathfrak i}\McNn$ denote the compact set of
all maximal ideals of $\maxspec \McNn$ containing $\mathfrak i$.
Then the map $ {\dot
{\mathcal Z}} $
of (\ref{equation:dotzeroset})
yields a
homeomorphism of $\maxspec_{\supseteq \mathfrak
i}\McNn$ onto the compact set $ {\mathcal Z}\mathfrak i\subseteq
[0,1]^{n}$.  The inverse of $\dot {\mathcal Z}$ is the map
$x\in {\mathcal Z}\mathfrak i\mapsto \mathfrak
m_{x} = \{f\in \McNn\mid f(x)=0\},$
and we have identical real numbers
\begin{equation}
\label{equation:zetapunto} f/\mathfrak m
= f({\dot {\mathcal Z}}(\mathfrak m)),
\,\,\,\,\forall f\in \McNn,\,\,\,\forall \mathfrak m \in
\maxspec_{\supseteq \mathfrak i}\in \McNn.
\end{equation}
\end{corollary}

\begin{proof} As a matter of fact, for each $x\in {\mathcal Z}\mathfrak
i$, $\mathfrak m_{x}$ is a maximal ideal of $\McNn.$
Further,  for each $
f\in \mathfrak i$, from  $f(x)=0$ we get $f\in \mathfrak m_{x},$ whence
$\mathfrak m_{x} \supseteq \mathfrak i$ and $\dot{\mathcal Z}\mathfrak
m_x=x.$ Let $\mathfrak p\in\maxspec_{\supseteq \mathfrak i}\McNn.$
Then $\mathcal Z \mathfrak p\subseteq\mathcal Z\mathfrak i$ and for
every $f\in \mathfrak p$ with $f(\dot{\mathcal Z}\mathfrak p)=0$ we
have $\mathfrak p\subseteq\mathfrak m_{\dot{\mathcal Z}(\mathfrak p)}$
and $\dot{\mathcal Z} \mathfrak p\in\mathcal Z \mathfrak i.$ The
assumed maximality of $\mathfrak p$ is to the effect that $\mathfrak p
=\mathfrak m_{\dot{\mathcal Z}(\mathfrak p)},$ whence $\dot{\mathcal
Z}$ is a one-one map from $\maxspec_{\supseteq \mathfrak i}\McNn$ onto
$\mathcal Z\mathfrak i$.  By definition of spectral topology,
$\dot{\mathcal Z}$ is a homeomorphism.  An application of Lemma
\ref{lemma:hoelder} now
settles the result.
\end{proof}

\begin{corollary}\label{corollary:HolII}
The quotient map $\kappa\colon \McNn\to \McNn/\mathfrak i$ determines
the  homeomorphism $ \mathfrak m\mapsto \mathfrak
m/\mathfrak i $ of $\maxspec_{\supseteq \mathfrak i}\McNn$ onto
$\maxspec \McNn/\mathfrak i$.  The inverse map is given by $
\kappa^{-1}(\mathfrak n)= \{f\in \McNn\mid f/\mathfrak i\in \mathfrak
n\}$ for each  $ \mathfrak n\in \maxspec \McNn/\mathfrak i.  $
\end{corollary}
\begin{proof}
{}The routine proof   follows by combining
  Lemma \ref{lemma:hoelder} with \cite[2.3.8]{bigkeiwol}.
\end{proof}

\subsection{Rational polyhedra and unimodular triangulations}
We will make use of a few elementary notions and techniques of
polyhedral topology. We refer to the first chapters of
\cite{ewa} for background.  By a {\it rational polyhedron} $P$
  in $\mathbb R^{n}$  we
  understand a finite union of simplexes $P=S_{1}\cup\cdots\cup S_{t}$
   in $\mathbb R^{n}$
  such that the coordinates of the vertices of every simplex $ S_{i}$
  are rational numbers.  For every simplicial complex $\Sigma$ the
  point-set union of the simplexes of $\Sigma$ is called the {\it
  support} of $\Sigma$ and is denoted $|\Sigma|$; $\Sigma$ is said to be
  a \emph{triangulation} of $|\Sigma|$.

  For any rational point $v \in
  \mathbb R^{n}$ the least common denominator of the coordinates of $v$
  is called the \emph{denominator} of $v$,  denoted $\den(v)$.
  The integer vector $\tilde v = \den(v)(v,1)\in \mathbb Z^{n+1}$ is
  called the {\it homogeneous correspondent} of $v$.  An $m$-simplex $U
  = \conv(w_{0}, \ldots,w_{m}) \subseteq [0,1]^{n}$ is said to be
  \emph{unimodular} if it is rational and the set of integer vectors
  $\{\tilde{w}_0, \ldots, \tilde{w}_m\}$ can be extended to a basis of
  the free abelian group ${\mathbb Z}^{n+1}$.  A simplicial complex is
  said to be a \emph{unimodular triangulation} (of its support) if all
  its simplexes are unimodular.

  As a remainder of the relevance of unimodular triangulations,
  recall that the homogeneous correspondent
  of a unimodular triangulation is known as a regular (or, nonsingular)
  fan \cite{ewa, oda}.

The following results show the connection among rational
polyhedra zero-sets of McNaughton maps and ideals in $\McNn$.

\begin{proposition}\cite[4.1,5.1]{marmun}
    \label{proposition:triangP}
Letting $P\subseteq \cube$, the following are equivalent:
\begin{itemize}
\item[(i)] $P$ is a rational polyhedron.
\item[(ii)] $P=|\Delta|$ for some unimodular
triangulation $\Delta$.
\item[(iii)] there is unimodular triangulation
$\nabla$ of $\cube$ such that $$P=\bigcup\{S\in
\nabla\colon S\subseteq P\}.$$
\item[(iv)] $P=\mathcal{Z}f$ for some $f\in\McNn$.
\end{itemize}
\end{proposition}

\begin{lemma}\label{lemma:zeroset}
Let $\mathfrak i$ be an ideal of $\McNn$. Then the following are equivalent:
\begin{itemize}
\item[(i)] $\mathfrak i$ is principal.
\item[(ii)] there exists $f\in\mathfrak i$
such that $\mathcal{Z}\mathfrak{i}=\mathcal{Z}f$.
\end{itemize}
\end{lemma}
\begin{proof}
For the non trivial direction, let $f\in\mathfrak i$ such that $\mathcal{Z}\mathfrak{i}=\mathcal{Z}f$.
It is no loss of generality to suppose $0\leq f$. We must verify that, for all
$0\leq g \in \McNn,$ $ g\in \mathfrak i\Leftrightarrow g\leq kf
\mbox{ for some } k=1,2,\ldots \,\,.$
 The $\leftarrow$-direction
follows from the fact that $f\in\mathfrak{i}$.  For the
$\rightarrow$-direction, let $\Lambda$, be a rational
triangulation of $\cube$, $f$ and $g$ are linear over each $S\in\Lambda$.
Let $\{v_1,\ldots,v_s\}$ be the vertices of $\Lambda$.
Since $\mathcal{Z}f=\mathcal{Z}\mathfrak{i}\subseteq\mathcal{Z}g$, $f(v_i)= 0$ implies $g(v_i)=0$.
Then there exists an integer $m_{i}>0$ such that $m_{i}f(v_i) \geq
f(v_i)$ for each $i=1,\ldots,s$.  Letting $m=\max(m_{1},\ldots,m_{s})$,
the desired result follows from the linearity of $f$ and $g$ over each
simplex of $\Lambda$.
\end{proof}

\section{Finitely presented unital $\ell$-groups and basis}\label{sec:finpresbasis}

\begin{theorem}
\label{theorem:basis}
A unital $\ell$-group $(G,u)$ is  finitely presented
iff it has a basis.
\end{theorem}

The $(\Rightarrow)$-direction of Theorem \ref{theorem:basis} is proved in \cite[5.2]{marmun}.
To prove the $(\Leftarrow)$-direction let $\mathcal{B}=\{b_{1},\ldots,b_{n}\}$ be a
basis of $(G,u)$, with multiplicities $m_{1},\ldots,m_{n}$.  Let
$\kappa \colon \McNn\rightarrow (G,u)$ be the unique unital
$\ell$-homomorphism extending the map $\pi_{i}\mapsto b_{i},$ as given
by Proposition \ref{proposition:free}.  Let the ideal $\mathfrak{i}$
of $\McNn$ be defined by $\mathfrak{i} = \ker(\kappa)$.
By Definition \ref{def:basis}(i),
$\kappa$ is onto $G$, thus
 \begin{equation}
     \label{equation:identification}
     (G,u)\cong \McNn/\mathfrak{i}.
     \end{equation}

For any $E\subseteq \mathcal B$ we define the simplex $ T_{E}\subseteq
  [0,1]^{n}$ by \begin{equation} \label{equation:ti} T_{E}=
  \conv\{e_{i}/m_{i}\mid b_{i}\in E\}, \end{equation} where $e_{i}$ is
  the $i$th standard basis vector of $\mathbb R^{n}$.
  {}From Definition \ref{def:basis}(ii) it follows  that
  $\kappa(\sum_{i} m_{i}\pi_{i})= \sum_{i} m_{i}\kappa(\pi_{i})= \sum_{i}
  m_{i}b_{i} =u$ whence,
  defining the function $a\in \McNn$ by
  $a=|1-\sum_{i} m_{i}\pi_{i}|$,

  \begin{equation}
 \label{equation:cl1}
 0\leq a \in \mathfrak i,  \mbox{ and }  {\mathcal Z}a=T_{\mathcal B}.
 \end{equation}

 \medskip
 \noindent
 Let $k=1,2,\ldots\,\,.$
 Then by a $k$-{\em cluster} of $\mathcal B$ we
 understand a $k$-element subset $C$
 of $\mathcal B$ such that
 $\,\, \bigwedge C \not=0$.  We denote by ${\mathcal B}^{\bowtie}$ the
 set of all clusters of $\mathcal B.$
For each $C\in \mathcal B^{\bowtie}$, displaying the
complementary set $\mathcal B\setminus C$  as
$\{b_{j_{1}},\ldots,b_{j_{s}}\}$, we
define the function $a_{C}\in \McNn$ by

  \begin{equation}
 \label{equation:duecroci}
a_{C}=\pi_{j_{1}}\vee\ldots\vee \pi_{j_{s}},
\,\,\,\mbox{($a_{C}=0$ in case $C=\mathfrak B$).}
 \end{equation}
We then  have
  \begin{equation}
 \label{equation:zeroaci}
T_{\mathcal B}\cap \mathcal Za_{C}=T_{C}.
 \end{equation}

\bigskip
\noindent
We next observe
 \begin{equation}
 \label{equation:cl2}
\bigwedge_{C\in \mathcal B^{\bowtie}} a_{C} \in \mathfrak i.
 \end{equation}
By (\ref{equation:duecroci}), the result is trivial if $\mathcal B$ is
a cluster in $\mathcal B^{\bowtie}$.  If this is not the case, let
$b_{i_C}\in\mathcal{B}\setminus C$ for each $C\in\mathcal
B^{\bowtie}$.  If   $D=\{b_{i_C}\colon C\in \mathcal
B^{\bowtie}\}\in\mathcal B^{\bowtie}$,  then $b_{i_D}\in D$, which is a
contradiction.  Therefore, $$\kappa(\bigwedge_{C\in \mathcal
B^{\bowtie}} \pi_{i_{C}})=\bigwedge_{C\in \mathcal B^{\bowtie}}
b_{i_{C}}=0,$$ i.e., $\bigwedge_{C\in \mathcal B^{\bowtie}}
\pi_{i_{C}}\in \mathfrak i$.  Since  each
$b_{i_C}\in\mathcal{B}\setminus C$ is arbitrary, (\ref{equation:cl2})
now follows from the distributivity of the underlying lattice of
$(G,u)$.

\bigskip
\noindent
Let the function $f^{*}\in \McNn$ be defined by
    \begin{equation}
    \label{equation:effestar}
    f^{*} = a
\vee \bigwedge_{C\in \mathcal B^{\bowtie}} a_{C}.
\end{equation}
{}From  (\ref{equation:cl1}) and (\ref{equation:cl2})
it follows that
\begin{equation}
     \label{equation:star-belongs}
    0\leq  f^{*}  \in \mathfrak i,
    \end{equation}
and from (\ref{equation:zeroaci}),
   \begin{equation}
       \label{equation:crossedsquare}
       {\mathcal Z}f^{*}= {\mathcal Z}a\,\,\,
       \cap \bigcup_{C\in \mathcal B^{\bowtie}}
       {\mathcal Z}a_{C} =
       \bigcup_{C\in \mathcal B^{\bowtie}}T_{C}.
       \end{equation}
{}From (\ref{equation:star-belongs})
we immediately have
   \begin{equation}
       \label{equation:dieci}
       {\mathcal Z}f^{*}  \supseteq {\mathcal Z} \mathfrak i.
       \end{equation}

\bigskip
To prove the converse inclusion,
for each cluster  $K$ of $\mathcal B$
we set
\begin{equation}\label{eq:defapo}
\apo (K)= \{\mathfrak n\in \maxspec \McNn/\mathfrak
i\mid   \mathfrak n
\supseteq \mathcal B\setminus K\}.
\end{equation}
For each $\mathfrak n\in \maxspec \McNn/\mathfrak
i$, letting $C_{\mathfrak n}$ be the cluster of all $b\in \mathcal B$
such that $b\notin \mathfrak n$, it follows that $\mathcal B\setminus
C_{\mathfrak n}\subseteq \mathfrak n$, whence $\mathfrak n\in \apo
(C_{\mathfrak n}).$ Thus, $\bigcup_{C\in \mathcal B^{\bowtie}} \apo
(C) \supseteq \maxspec \McNn/\mathfrak i$.  Since the converse
inclusion holds by definition, we have
 \begin{equation}
\label{equation:cross2squares}
\maxspec \McNn/\mathfrak i=\bigcup_{C\in \mathcal B^{\bowtie}}
\apo (C).
\end{equation}
For each $K\in\mathcal{B}^{\bowtie}$, we denote by $\apor (K)$ the
inverse image of $\apo (K)$ under the composition of the
homeomorphisms $x\mapsto \mathfrak m_{x}\mapsto \mathfrak
m_{x}/\mathfrak i $ of Corollaries \ref{corollary:HolI} and
\ref{corollary:HolII}, where $m_{x} = \{f\in \McNn\mid f(x)=0\}$.  In
other words, \begin{equation} \label{equation:apor} \apor (K)=\{x\in
{\mathcal Z}\mathfrak i\mid {\mathfrak m_{x}}/{\mathfrak i} \in \apo
(K)\}.  \end{equation} {}From
(\ref{equation:crossedsquare})-(\ref{equation:cross2squares}) we get
\begin{equation} \label{equation:liftata} \bigcup_{C\in \mathcal
B^{\bowtie}} \apor (C) = {\mathcal Z}\mathfrak i \subseteq \mathcal
Zf^{*} = \bigcup_{C\in \mathcal B^{\bowtie}}T_{C}.  \end{equation}
This inclusion can be refined as follows:

\medskip

\noindent {\it Claim 1:} For each cluster $C$
of $\mathcal B$, $\apor (C)\subseteq T_{C}.$

\medskip
As a matter of fact,
by (\ref{eq:defapo}) and condition (iii) of Definition \ref{def:basis} we have
\begin{eqnarray}
\label{equation:apo}
 \nonumber  \apo (C) &=& \{\mathfrak n\in \maxspec \McNn/\mathfrak
i\mid b/\mathfrak n=0,\forall b\in\mathcal B\setminus C\} \\
  &=& \{\mathfrak n\in \maxspec \McNn/\mathfrak
i\mid \frac{m_{i_{1}}b_{i_{1}}+\cdots+m_{i_{t}}b_{i_{t}}}{\mathfrak n}=1\},
\end{eqnarray}
for each cluster $C=\{b_{i_{1}},\ldots,b_{i_{t}}\}$ of $\mathcal B$.
By Lemma \ref{lemma:hoelder}, for each
$\mathfrak m\in \maxspec_{\supseteq \mathfrak i}\McNn$ the unital
$\ell$-group $\frac{\McNn}{\mathfrak m}$ and its isomorphic copy
$\frac{\McNn/\mathfrak i}{\mathfrak m/\mathfrak i}$ are canonically
isomorphic to the same unital $\ell$-subgroup of $(\mathbb R,1)$.
Thus for each $f\in \McNn$ we have identical real numbers $
\frac{f/\mathfrak i}{\mathfrak m/ \mathfrak i} =\frac{f}{\mathfrak
m}.$ Thus, by Corollary \ref{corollary:HolI} and Corollary \ref{corollary:HolII}
    	\begin{equation}
	    \label{equation:alpha-beta}
	      \frac{f/\mathfrak i}{\mathfrak n/\mathfrak i}
	      =
	      f({\dot {\mathcal Z}}(\kappa^{-1}(\mathfrak n))),
	      \,\,\,\forall
	      \mathfrak n\in \maxspec \McNn/\mathfrak i,
	    \end{equation}
or equivalently,
\begin{equation}
	    \label{equation:beta-alpha}
	    f(x)=\frac{f/\mathfrak i}{\mathfrak m_{x}/\mathfrak i}
	    ,\,\,\,
	    \forall x\in {\mathcal Z}\mathfrak i.
	    \end{equation}
Combining (\ref{equation:apor}) with (\ref{equation:apo}), we obtain
$(y_{1},\ldots,y_{n})\in \apor (C)$ if and only if
$$\frac{m_{i_{1}}b_{i_{1}}+\cdots+m_{i_{t}}b_{i_{t}}}{\mathfrak{m}_x/\mathfrak{i}}
=\frac{(m_{i_{1}}y_{i_{1}}+
\cdots+m_{i_{t}}y_{i_{t}})/\mathfrak{i}}{\mathfrak{m}_x/\mathfrak{i}}=1.$$
Now recalling (\ref{equation:ti}), by (\ref{equation:beta-alpha}) we
obtain \begin{equation} \label{equation:lorena-apor} \apor
(C)=\{(y_{1},\ldots,y_{n}) \in {\mathcal Z}\mathfrak i\mid
m_{i_{1}}y_{i_{1}}+\cdots+m_{i_{t}}y_{i_{t}}=1\} \subseteq T_{C},
\end{equation} thus settling Claim 1.

\bigskip
Actually, a stronger result holds:
\medskip

\noindent
{\it Claim 2.} For every cluster $C$
of $\mathcal B$,
$\apor (C)= T_{C}.$

\medskip
The proof is by induction on the
number $l$  of elements of $C$.

\medskip \noindent{\it Base case: $l=1$.}
Then for a unique $j\in
\{1,\ldots,n\}$ we have $C=\{b_{j}\}=\{\pi_{j}/\mathfrak i\}$.
Condition (ii) in Definition \ref{def:basis} is to the effect that
$\apo ({C})$ contains exactly one element $\mathfrak n$.  By Lemma
\ref{lemma:hoelder}, $\mathfrak n$ is the only maximal ideal of
$\McNn/\mathfrak i$ such that $ 0 = b/\mathfrak n$ for all $b\not=
b_{j}; $ by (\ref{equation:apo}), $\mathfrak n$ is uniquely determined
by the condition $ 1= {m_{j}b_{j}}/{\mathfrak n}=
({m_{j}\pi_{j}/\mathfrak i})/{\mathfrak n}.$ Letting $z\in {\mathcal
Z}\mathfrak i$ be the image of $\mathfrak n$ in $\apor (C)$, by
(\ref{equation:ti}) and Claim 1 we have
$z=e_{j}/m_{j}.$ We conclude that $\apor (C) = \{e_{j}/m_{j}\}=
\conv\{e_{j}/m_{j}\}= T_{C}.  $

\bigskip \noindent{\it Induction Step, $l+1$.} Let us write
$C=\{b_{i_{0}},\ldots,b_{i_{l}}\}$.  Since every $l$-element subset
$C'$ of $C$ is a cluster of $\mathcal B$, by induction hypothesis
$\apor (C') = T_{C'} .  $ $T_{C'}$ is known as a {\it facet} of
$T_{C}$.  By Claim 1, $\apor (C)$ is a nonempty
subset of $T_{C}$ containing all facets of $T_{C}$.  Further,
$\apor (C)$ is homeomorphic to an $l$-simplex,
because so is its homeomorphic
copy $\apo (C)$, by condition (ii) in Definition \ref{def:basis}.
Observe that $T_{C}$ is {\it contractible}  (i.e.,
$T_{C}$ is continuously shrinkable to a point).
By way of contradiction, suppose
$\apor (C)$ is a proper subset of $T_{C}$. Then a classical result
in algebraic topology shows that
$\apor (C)$ is not contractible.  Thus
$\apor (C)$ is not homeomorphic to any $l$-simplex, a contradiction
showing $\apor (C)=T_{C},$
and settling  Claim 2.

\bigskip
\noindent
{}Combining  Claim 2 and
(\ref{equation:liftata}),
we can write
  \begin{equation}
\label{equation:circles}
{\mathcal Z}f^{*}= \bigcup_{C\in \mathcal B^{\bowtie}}
T_{C}={\mathcal Z}\mathfrak i.
\end{equation}

\medskip

   Recalling Lemma \ref{lemma:zeroset} it follows
   that $\mathfrak i$ is the ideal generated by $f^{*}$.
By (\ref{equation:identification}),
$(G,u)$ is finitely presented.
The proof of
Theorem \ref{theorem:basis}
is thus complete.
\hfill $\Box$

\section{A class of projective unital $\ell$-groups}\label{sec:proj}
In Theorem \ref{theorem:projective} below we will
construct  a large class of
projective unital $\ell$-groups.
For the proof we prepare

\begin{lemma}\label{lemma:LinearMap}
Let  $S=
     \conv(x_{1},\ldots,x_{k})
     \subseteq [0,1]^{n}$ be a
     unimodular $(k-1)$-simplex and $v\in\{0,1\}^n$ a vertex of $\cube$.
     Then for every $Y\subseteq \{x_1,\ldots,x_k\}$ there is a
	matrix  $M\in\Zed^{n\times n}$ and  a
	vector  $b\in\Zed^{n}$  such that
\begin{equation}\label{eq:zedmap}
M x_{i}+b_{i}=\left\{\begin{tabular}{ll}
$v$& if $x_i\in Y$,\\
$x_i$ & otherwise.
\end{tabular}\right.
\end{equation}
  \end{lemma}
\begin{proof}  Since
$S$ is  unimodular,
 the set $\{\tilde{x_{1}},\ldots, \tilde{x_{k}}\}$
 of homogeneous
correspondents
of
$x_{1},\ldots,x_{k}$
can be extended to a basis
$
\{\tilde{x_{1}},\ldots, \tilde{x_{k}}, q_{k+1},\ldots,
q_{n+1}\}
$
of the free abelian group $\Zed^{n+1}$.  The
$(n+1)\times(n+1)$ matrix $D$ with column vectors
$\tilde{x_{1}},\ldots,\tilde{x_{k}}, q_{k+1},\ldots, q_{n+1}$
is invertible and $D^{-1}\in\Zed^{(n+1)\times (n+1)}.$
For each $i=1,\ldots,t$ let  $c_{i}\in
\Zed^{n+1}$ be defined  by
$$
 c_{i}=\left\{\begin{tabular}{ll}
       $\den(x_i)(v,1)$ & if $x_i\in Y$, \\
       $\tilde{x_{i}}$ & otherwise.
       \end{tabular}\right.
$$
 Let $C\in\Zed^{(n+1)\times (n+1)}$ be the matrix whose
columns are given by the vectors $c_{1},\ldots,c_{k}, q_{k+1},\ldots,
q_{n+1}$.  Since $D$ and $C$ have the same $(n+1)$th row,
     $$
       C D^{-1}=\left(\begin{tabular}{c|c}
       $M$ & $ d$ \\
       \hline
       $0,\ldots,0 $ & $1$
       \end{tabular}\right)
     $$
     for some  $n\times n$ integer matrix $M$
     and   vector $ d\in \mathbb Z^{n}$.
     For each  $i=1,\ldots,k$ we then have
     $
     (C D^{-1}) \tilde{x_{i}}=(C
     D^{-1})\den(x_{i})(x_{i},1)=\den(x_{i})(Mx_{i}+ d,1).
     $
     By definition,
     $(C D^{-1}) \tilde{x_{i}}=c_{i}=\den(x_{i})(v,1)$ if
     $x_i\in Y$ and $(C D^{-1}) \tilde{x_{i}}=\tilde{x_{k}}=\den(x_{i})(x_i,1)$
     otherwise. Thus (\ref{eq:zedmap}) is satisfied.
\end{proof}

 \begin{theorem}
     \label{theorem:projective}
     Suppose the unital $\ell$-group $(G,u)$
     has a basis  $\mathcal B$ with
     $\bigwedge \mathcal B\not= 0.$ Suppose
     at least one of the multiplicities of $\mathcal B$
     is equal to $1$.
	 Then $(G,u)$  is projective.
     \end{theorem}

\begin{proof}
By assumption, $\mathcal B$ itself is a basis of
$(G,u)$ with multiplicities $1=m_{1}\leq m_{2}\leq\ldots\leq m_{n}.$
We keep the notation of the proof of Theorem \ref{theorem:basis}.  In
particular, $T_{\mathcal
B}=\conv(e_{1},e_{2}/m_{2},\ldots,e_{n}/m_{n})$, where, as the reader
will recall, $e_{i}$ denotes the $i$th basis vector in $\mathbb
R^{n}$.
Proposition \ref{proposition:triangP}
yields a
unimodular triangulation $\Delta$ of $\cube$ such that $T_{\mathcal B}$ is a
union of simplexes of $\Delta$, and all vertices of
(every simplex of) $\Delta$ have rational coordinates.

We next define  the function ${\bf f}
\colon [0,1]^{n}\to [0,1]^{n}$ by stipulating that,
for each vertex $v$ of $\Delta$,
\begin{equation}
    \label{equation:retraction}
      {\bf f}(v)= \left\{
         \begin{tabular}{ll}
            $v$  & if $v\in  T_{\mathcal B} $ \\[1mm]
            $e_{1}$ & if  $v \not\in  T_{\mathcal B} $
         \end{tabular}
      \right.
\end{equation}
    and ${\bf f}$ is linear over each
    simplex of $\Delta$.
Then  $\bf f$ is a continuous map
    and ${\bf f}\restrict T_{\mathcal B}$
    is the identity map on $T_{\mathcal B}.$
For any simplex $S$ of $\Delta$, let $\partial S$ denote the set of
extremal points of $S$.  Since ${\bf f}$ is linear over $S$ and ${\bf
f}(v)\in T_{\mathcal B}$ for each $v\in\partial S$, we have $ {\bf
f}(S)={\bf f}(\conv(\partial S))=\conv({\bf f}(\partial S)) \subseteq
\conv(T_{\mathcal B})=T_{\mathcal B}, $ whence
\begin{equation}
\label{equation:ontoTb}
{\bf f}([0,1]^n)=T_{\mathcal B}.
\end{equation} We have thus shown that ${\bf f}\circ {\bf f}={\bf f}$
and ${\bf f}$ is a continuous retraction of
$\cube$  onto $T_{\mathcal B}$  which is linear
on each simplex of $\Delta$.

 By Lemma \ref{lemma:LinearMap}, the
coefficients of each linear piece of $\bf f$
are integers. Therefore, the map $\varphi\colon\McNn\to \McNn$ given by
    \begin{equation}\label{equation:defvarphi}
    \varphi(g)=
    g\circ {\bf f}.
    \end{equation}
is well defined. It follows straightforwardly that $\varphi$ is a
unital $\ell$-homomorphism.
Since ${\bf f}\circ {\bf f}={\bf f}$ then $\varphi \circ
\varphi=\varphi.$ In other words, $\varphi$ is an idempotent
endomorphism of $\McNn$.  Stated otherwise, the unital $\ell$-subgroup
$\varphi(\McNn)$ of $\McNn$ is a retraction of $\McNn$.  Applying now
the universal property of $\McNn$, (Proposition
\ref{proposition:free}) one sees that $\McNn$ is projective.  A
routine exercise using the fact that $\varphi(\McNn)$ is a retraction
of $\McNn$ shows that $\varphi(\McNn)$ is projective.

To conclude the proof it is enough to show that $\varphi(\McNn)$ is
unitally $\ell$-isomorphic to $(G,u)$.  In proving the
$(\Leftarrow)$-direction of Theorem \ref{theorem:basis} we have seen
that $(G,u)$ is unitally $\ell$-isomorphic to $\McNn/\mathfrak i$, for
some ideal $\mathfrak i$ having following characterization:
$$
\mathfrak i=\left\{g\in\McNn\mid\mathcal Z
g\supseteq\bigcup_{C\in \mathcal
B^{\bowtie}}T_{C} \right\} =\{g\in\McNn\mid\mathcal
Z g\supseteq T_{\mathcal
B}\}.
$$
Letting $\ker(\varphi)$ be the kernel of $\varphi$, by
(\ref{equation:ontoTb}) and (\ref{equation:defvarphi}) we have
$$
  g\in\ker(\varphi)\Leftrightarrow g\circ{\bf f}=0
\Leftrightarrow g({\bf f}([0,1]^n))=\{0\}
\Leftrightarrow g(T_{\mathcal B})=\{0\}
\Leftrightarrow \mathcal Z g\supseteq T_{\mathcal B}
\Leftrightarrow g\in\mathfrak i.
$$
Therefore,
$
(G,u)\cong\McNn/\mathfrak i=\McNn/\mathfrak
\ker(\varphi)\cong\varphi(\McNn),
$
and the proof is complete.
\end{proof}

\section{The underlying dimension group of
a unital $\ell$-group with a basis}
In the category $\mathcal P$ of {\it partially ordered abelian groups
with order-unit}, \cite[p.12]{goo2} objects are pairs $(G,u)$, where
$G$ is a partially ordered abelian group and $u$ is an order-unit of
$G$.  A morphism $\phi\colon (G,u)\to (H,v)$ of $\mathcal P$ is a
{\it unital} (i.e., unit-preserving) {\it positive} (in the sense that
$\phi(G^{+})\subseteq H^{+}$) homomorphism.

Following \cite[p.47]{goo2}, by a {\it unital simplicial} group we
understand an object of $\mathcal P$ that is isomorphic (in $\mathcal
P$) to the free abelian group
$\mathbb Z^{n}$ for some integer $n>0$ equipped with the
product ordering: $(x_{1},\ldots,x_{n})\geq 0\,\,$ iff $\,\,x_{i}\geq
0\,\,\,\forall i=1,\ldots,n.$

A {\it unital dimension group} $(G,u)$ is an object of $\mathcal P$
such that $G=G^{+}-G^{+}$, sums of intervals are intervals, and $kg\in
G^{+}\Rightarrow g\in G^{+},$ for any $g\in G$ and  integer
$k>0$.  For short, $G$ is directed, Riesz, and unperforated,
\cite[p.44]{goo2}.  In \cite[\S 2]{fuc-pisa} one can find several
characterizations of the Riesz property.  By Elliott classification
theory \cite{goo1}, countable unital dimension groups are complete
classifiers of AF algebras, the norm limits of ascending sequences of
finite-dimensional C$^{*}$-algebras, all with the same unit.

Given a unital $\ell$-group $(G,u)$ let $(G,u)_{\dim}$ denote the
underlying group of $(G,u)$ equipped with the same positive cone
$G^{+}$ and order-unit $u$, but forgetting the lattice structure of
$(G,u)$.  Then $(G,u)_{\dim}$ is a unital dimension group.  Thus in
particular, every unital simplicial group is a unital dimension group.
Since the properties of directedness, Riesz, and unperforatedness are
preserved by direct limits, then direct limits of unital simplicial
groups are unital dimension groups.

The celebrated Effros-Handelman-Shen theorem \cite{effhanshe},
\cite[3.21]{goo2} (also see Grillet's theorem \cite[2.1]{gri} in the
light of \cite[Remark 3.2]{gooweh}) states the converse: for every
unital dimension group $(G,u)$ we can write
$$
    (G,u) \cong\lim \{\phi_{ij}
    \colon
    (\mathbb Z^{n_{i}}, u_{i})\to
    (\mathbb Z^{n_{j}}, u_{j})\mid
    i,j\in I\}
$$
 for some direct system of unital simplicial groups and unital positive
    homomorphisms in $\mathcal P$.  For dimension groups of the form
    $(G,u)_{\dim}$, with $(G,u)$ a unital $\ell$-group, Marra \cite{mar}
proved that the maps $\phi_{ij}$ can be assumed to be 1-1.

A further simplification occurs when $(G,u)$ has a basis: as a matter
of fact, in Theorem \ref{theorem:system} below we will prove that the
set of bases of $(G,u)$ is rich enough to provide a direct system of
unital simplicial groups and 1-1 unital homomorphisms such that
$(G,u)_{\dim}$ is the limit of this system in the category $\mathcal
P$.  To this purpose, given a basis $\mathcal
B=\{b_{1},\ldots,b_{n}\}$ of a unital $\ell$-group $(G,u)$, we let $$
\grp\mathcal B = \mathbb Z b_{1}+\cdots+\mathbb Z b_{n} $$ denote the
group generated by ${\mathcal B}$ in (the underlying group of) $G$.
Similarly, $$ \sgr \mathcal B = \mathbb Z_{\geq 0}
\,b_{1}+\cdots+\mathbb Z_{\geq 0} \,b_{n} $$ will denote the semigroup
generated by ${\mathcal B}$ together with the zero element.

Assuming, as we are doing throughout the rest of this paper, that the
elements of $\mathcal B$ are listed in some prescribed order, by
definition of $\mathcal B$ the $n$-tuple of multiplicities
$m_{\mathcal B}=(m_{1},\ldots,m_{n})$ is uniquely determined by the
$n$-tuple $(b_{1},\ldots,b_{n})$.

\begin{proposition}
\label{proposition:basic}
Let $\mathcal B=\{b_{1},\ldots,b_{n}\}$ be a
basis
of a unital
$\ell$-group $(G,u)$.
Let
  $$ G_{\mathcal B} =
(\grp\mathcal B, \,\,\,\sgr \mathcal B,\,\, u)$$
denote the group $\grp\mathcal B$
equipped with the positive
cone $\sgr \mathcal B$
and
with the distinguished element $u=
\sum m_{i}b_{i}$.  Let
$$\mathbb Z_{\mathcal B}=
(\mathbb Z^{n},{({\mathbb Z}^{+}})^{n}, m_{\mathcal B})
$$
be the standard simplicial group of rank $n$, with the
$n$-tuple $m_{\mathcal B}$ as
a distinguished element.
Then
\begin{enumerate}
\item $\mathcal B$ is a free generating set of
the free abelian group $\grp\mathcal B$ of rank $n$.

\smallskip
\item $G^{+}\cap \grp\mathcal B=\sgr \mathcal B.$

\smallskip
\item The map $b_{i}\mapsto e_{i}$ uniquely extends
to an isomorphism
$\psi_{\mathcal B}\colon \grp_{\mathcal
B}\cong \mathbb Z^{n}$.

\smallskip
\item
$\psi_{\mathcal B}$
is in fact an isomorphism (in the category
$\mathcal P$) of $G_{\mathcal B}$  onto
$\mathbb Z_{\mathcal B}$,  whence
$G_{\mathcal B}$  is a unital simplicial group,
called the {\rm basic group of $\mathcal B$};
further,
$\mathcal B$ is the set of atoms (=minimal positive nonzero
elements)  of $G_{\mathcal B}$;
thus if $\mathcal B' \not= \mathcal B$ is
another basis of $(G,u)$ then $G_{\mathcal B} \not=
G_{\mathcal B'}$.
\end{enumerate}
\end{proposition}

 \begin{proof}
(1) By condition (ii) in the definition of $\mathcal B$,
no nonzero linear combination of the elements of
$\mathcal B$ is zero in (the $\mathbb Z$-module)
$G$.  It is well known
that $G$ is torsion-free.  Thus $\mathcal B$ is a free generating set
in $\grp\mathcal B$, and $\grp \mathcal B$ is free abelian of rank
$n$.

To prove (2),  suppose $g\in G^{+}\cap \grp\mathcal B$, and write
$g=\sum_{i=1}^{n}l_{i}b_{i}$ for suitable integers
$l_{1},\ldots,l_{n}$.  Fix now $j\in \{1,\ldots,n\}$ and let
$\mathfrak n_{j}$ be the only maximal ideal of $G$ such that $b_{k}\in
\mathfrak n_{j}$ for all $k\not=j,$ as given by condition (ii) in the
definition of $\mathcal B$.  By condition (iii) we have $$ 0\leq
\sum_{i=1}^{n}l_{i}b_{i} \Rightarrow 0\leq
\frac{\sum_{i=1}^{n}l_{i}b_{i}}{\mathfrak n_{j}} =
\frac{l_{j}b_{j}}{\mathfrak n_{j}} = \frac{l_{j}}{m_{j}}, $$ whence
$0\leq l_{j}$ for all $j$, and $g\in \sgr\mathcal B.$ The converse
inclusion is trivial.

To prove (3) it is enough to note that the map $b_{i}\mapsto e_{i}$ is
a one-one correspondence between the free generating set $\mathcal B$
of $\grp\mathcal B$ and the free generating set
$\{e_{1},\ldots,e_{n}\}$ of $\mathbb Z^{n}$.

Concerning (4).  It is easy to see that
$\mathcal B$ is the set of atoms of
$G_{\mathcal B}$, and $\{e_{1},\ldots,e_{n}\}$ is the set of atoms of
the simplicial group $(\mathbb Z^{n},{({\mathbb Z}^{+}})^{n})$.  Thus
$\psi_{\mathcal B}$ is an isomorphism of $G_{\mathcal B}$ onto
$(\mathbb Z^{n},{{\mathbb Z}^{+}}^{n})$, and $G_{\mathcal B}$ is
simplicial.  Trivially, $\psi_{\mathcal B}$ preserves the order-unit.
So $G_{\mathcal B}$ is a unital simplicial group which is isomorphic
(in $\mathcal P$) to $\mathbb Z_{\mathcal B}.$
The rest is clear.  \end{proof}

\bigskip \noindent Given two bases $\mathcal B'$ and $\mathcal B$ of a
unital $\ell$-group $(G,u)$ we say that $\mathcal B'$ {\it refines}
$\mathcal B$ if $ \mathcal B\subseteq \sgr \mathcal B'.  $

\medskip

{}From the foregoing proposition we immediately obtain.

\begin{proposition}
    \label{proposition:refine-to-map}
  Let
 $\mathcal B'=\{b'_{1},\ldots,b'_{n'}\}$ and
 $\mathcal B=\{b_{1},\ldots,b_{n}\}$ be bases of a
 unital $\ell$-group
     $(G,u)$ such that
 $\mathcal B'$ refines
$\mathcal B$.
We then have:
\begin{enumerate}
\item For each $i=1,\ldots,n$,
the element
$b_{i}$ is
expressible as a linear combination
$b_{i}=m_{1i}b'_{1}+\cdots+m_{n'i}b'_{n'}\,$, for
uniquely determined integers
$m_{ki}\geq 0,\,\,\,(k=1,\ldots,n')$.

\smallskip
\item The  $n'\times n$ integer matrix
$M_{\mathcal B\mathcal B'}$ whose entries are the $m_{ki}$,
has rank equal to $n$.

\smallskip
\item The inclusion map
$G_{\mathcal B}\to G_{\mathcal B'}$
induces the
unital positive $1$-$1$ homomorphism
$$
\phi_{\mathcal B\mathcal B'}\colon
(y_{1},\ldots,y_{n}) \in \mathbb Z^{n}
\mapsto (z_{1},\ldots,z_{n'})=
M_{\mathcal B\mathcal B'}(y_{1},\ldots,y_{n}) \in \mathbb
Z^{n'}
$$
of $(\mathbb Z_{\mathcal B},m_{\mathcal B})$ into
$(\mathbb Z_{\mathcal B'},m_{\mathcal B'})$,
and we have a commutative diagram

\begin{equation}
    \label{equation:diagram}
\def\normalbaselines{\baselineskip20pt
\lineskip3pt
\lineskiplimit3pt}
\def\mapright#1{\smash{\mathop{\longrightarrow}\limits^{#1}}}
\def\mapdown#1{\Big\downarrow\rlap{$\vcenter{\hbox{$\scriptstyle#1$}}$}}
\begin{matrix}
G_{\mathcal B} &\mapright{\rm inclusion}&G_{\mathcal B'}\cr
\mapdown{\psi_{\mathcal B}}& &\mapdown{\psi_{\mathcal B'}}\cr
     (\mathbb Z_{\mathcal B},m_{\mathcal B}) &\mapright{{\phi_{\mathcal B
     \mathcal B'}}}&  (\mathbb Z_{\mathcal B'},m_{\mathcal B'})\cr
\end{matrix}
\end{equation}
\end{enumerate}
\end{proposition}

\bigskip \begin{theorem}
\label{theorem:system} Suppose the unital
$\ell$-group $(G,u)$ has a basis.
We then have:

 \begin{enumerate}
\item Any two basic groups $G_{\mathcal B}, G_{\mathcal F}$ of $(G,u)$
are jointly embeddable (by unit preserving, order preserving
inclusions) into some basic group $G_{\mathcal B'}$ of $(G,u)$.

\smallskip
 \item We then have a direct system $ \{\phi_{\mathcal B\mathcal B'}
 \colon (\mathbb Z_{{\mathcal B}},m_{\mathcal B}) \to (\mathbb
 Z_{{\mathcal B'}},m_{\mathcal B'})\} $ of unital simplicial groups and
 unital positive $1$-$1$ homomorphisms in $\mathcal P$, indexed by all
 pairs $\mathcal B,\mathcal B'$ of bases of $(G,u)$ such that $
 \mathcal B \subseteq \sgr \mathcal B'$.

\smallskip
 \item Further, $ \lim \,\, \{\phi_{\mathcal B\mathcal B'}
\colon (\mathbb Z_{{\mathcal B}},m_{\mathcal B}) \to (\mathbb
Z_{{\mathcal B'}},m_{\mathcal B'})\} \cong (G,u)_{\dim}\,.$
\end{enumerate}

\end{theorem}

\begin{proof}
(1) By Theorem \ref{theorem:basis}, $(G,u)$ is finitely presented, and
for some $n=1,2,\ldots$ we have \begin{equation}
\label{equation:finpres} (G,u)\cong\McNn/\mathfrak j,\,\,\,\mbox{ for
some principal ideal $\mathfrak j$ of }\McNn.
\end{equation}
Suppose
$\mathfrak j$ is generated by $f \in \McNn$.  Recalling the
notation $\mathcal Zf$ for the zeroset of $f$, a variant of
\cite[5.2]{gla} shows that $\McNn/\mathfrak
j\cong\McNn\restrict\mathcal Zf$.  A fortiori, $(G,u)$ is archimedean.
{}From the abstract De Concini-Procesi lemma \cite[5.4]{marmun} it
follows that $\mathcal B$ and $\mathcal F$ have a joint refinement
$\mathcal B'$.  Direct inspection of the proof therein, shows that
$\mathcal B'$ is obtained from $\mathcal B$ by finitely many
applications of the following operation: replace a 2-cluster $\{b,c\}$
of a basis $\mathcal A$, by the three elements $b\wedge c, b-(b\wedge
c), c-(b\wedge c).$ The result is a basis $\mathcal A'$ such that
$\mathcal A\subseteq \sgr \mathcal A'.$ Thus $\mathcal B\subseteq \sgr
\mathcal B'$.  The desired conclusion now follows from Proposition
\ref{proposition:refine-to-map}.

\smallskip
The proof of (2)   now immediately
follows from Proposition \ref{proposition:refine-to-map}.

\smallskip
Concerning (3),
in view of (\ref{equation:diagram}) it is sufficient to prove that
$G= \bigcup\{\grp \mathcal B\mid \mathcal B \mbox{ a basis of
}(G,u)\}$ and that $G^{+}= \bigcup\{\sgr \mathcal B\mid \mathcal B
\mbox{ a basis of }(G,u)\}.  $ Since $G=G^{+}-G^{+},$ only the latter
identity must be proved.  In other words, we must prove:
\begin{equation}
 \label{equation:final}
 \mbox{For every } p\in G^{+},\,\,(G,u)
 \mbox{ has a basis
    $\mathcal B$   such that
    } p\in \sgr \mathcal B.
    \end{equation}
As remarked above, we have a unital $\ell$-isomorphism $ \omega\colon
(G,u)\cong \McNn\restrict \mathcal Zf.  $ By \cite[4.5]{marmun},
$\omega$ induces a 1-1 correspondence between bases of the archimedean
unital $\ell$-group $(G,u)$ and Schauder bases $\mathcal H_{\Delta}$
of $\McNn\restrict \mathcal Zf$, where $\Delta$ ranges over unimodular
triangulations of the rational polyhedron $\mathcal Zf$.  Trivially, $
\mathcal B\subseteq \sgr \mathcal B'$ iff  $ \omega(\mathcal
B)\subseteq \omega(\mathcal B').  $ Thus (\ref{equation:final}) boils
down to proving that for every $0\leq g \in \McNn\restrict \mathcal
Zf$ there is a unimodular triangulation $\Delta$ of $\mathcal Zf$ such
that $g \in \sgr \mathcal H_{\Delta}.$ Let $\Delta$ be a unimodular
triangulation of $\mathcal Zf$ such that $g$ is linear over every
simplex of $\Delta$.
The existence of $\Delta$ is ensured by \cite[1.2]{mun88}.  Since
every linear piece of $g$ has integer coefficients, for each vertex
$v$ of $\Delta$ we can write $ g(v)={n_{v}}/{\den(v)}\,\,\, \mbox{ for
some } 0\leq n_{v}\in \mathbb Z.  $ As in the final
part of the proof of Theorem
\ref{theorem:basis}, let $h_{v}\colon |\Delta|\to\mathbb R$
denote the Schauder hat of $\Delta$
at $v$.  Let the function
$\overline{g}\in \sgr \mathcal
H_{\Delta}
\subseteq \McNn\restrict \mathcal Zf$
be defined by $$\overline{g}=\sum\{n_{v}h_{v}\mid v \mbox{ a vertex of
}\Delta\}.$$ Then $\overline{g}(v)=g(v)$ for each vertex $v$ of
$\Delta$ and $\overline{g}$ is linear over each simplex of $\Delta.$
It follows that
$\overline{g}=g$, which completes the proof.  \end{proof}


\end{document}